\documentclass[a4paper]{amsart}
\usepackage{fullpage,amssymb,mathtools,verbatim,hyperref}
\newtheorem{theorem}{Theorem}
\newtheorem{lemma}[theorem]{Lemma}
\newtheorem{proposition}[theorem]{Proposition}

\newtheorem{definition}[theorem]{Definition}
\newtheorem{conjecture}[theorem]{Conjecture}
\theoremstyle{remark}

\numberwithin{theorem}{section}
\numberwithin{equation}{section}
\newcommand{\C}{\mathbb{C}}

\newcommand{\Q}{\mathbb{Q}}
\newcommand{\Z}{\mathbb{Z}}
\newcommand{\N}{\mathbb{N}}
\newcommand{\F}{\mathbb{F}}
\newcommand{\OK}{\mathcal{O}_K}
\newcommand{\GEM}{\mathrm{GEM}}
\newcommand{\pp}{\mathfrak{p}}
\newcommand{\qq}{\mathfrak{q}}
\DeclareMathOperator{\sgn}{sgn}
\DeclareMathOperator{\lcm}{lcm}
\DeclareMathOperator{\Disc}{Disc}
\DeclareMathOperator{\Gal}{Gal}
\begin{document}
\title{A generalisation of the Euclid--Mullin sequences}
\author{Andrew R. Booker}
\email{andrew.booker@bristol.ac.uk}
\author{Omri Simon}
\email{OmriSimon@outlook.com}
\address{School of Mathematics, University of Bristol,
Woodland Road, Bristol, BS8 1UG}
\begin{abstract}
We extend Mullin's prime-generating procedures to produce sequences
of primes lying in given residue classes.  In particular we study
the sequences generated by cyclotomic polynomials $\Phi_m(cx)$ for
suitable $c\in\Z$.  Under the Extended Riemann Hypothesis in general and
unconditionally for some moduli, we show that the analogue of the second
Euclid--Mullin sequence omits infinitely many primes $\equiv1\pmod{m}$.
We further show unconditionally that at least one prime is omitted
for infinitely many $m$. This generalises work of the first author for $m=1$
and the second author for $m=2^k$.
\end{abstract}

\maketitle

\section{Introduction}
In \cite{Mullin}, Mullin constructs two sequences of prime
numbers based on Euclid's proof of infinitude of primes.
For Mullin's first sequence, say $\{p_n\}_{n=1}^\infty$,
define $p_n$ to be the smallest prime dividing $p_1\cdots p_{n-1}+1$
(with the convention that the empty product is $1$, so $p_1=2$).
Mullin's second sequence, $\{P_n\}_{n=1}^\infty$, is defined
analogously, choosing instead the largest prime at each stage.
Mullin then asked whether every prime is contained in each of
these sequences, and if not, whether they are recursive, i.e.\
whether there is an algorithm to determine membership
of a given prime.

Cox and van der Poorten \cite{CvdP} showed that the second sequence omits
all primes less than $47$, apart from the first four terms $2$, $3$, $7$,
and $43$.\footnote{Their paper claimed to show that $47$ is omitted as
well, but contained a numerical error that went unnoticed for decades;
we thank Thorkil Naur for bringing this to our attention. With the terms
of the sequence that were known at the time, the Cox--van der Poorten
method only works to rule out primes below $47$, and it only became
possible to rule $47$ out with the computation of the $12$th term by
Wagstaff \cite{Wagstaff} in 1991.} Their method can be extended to the
primes less than $79$ using more known terms of $\{P_n\}_{n=1}^\infty$
\cite{Guy-Nowakowski,Naur,Wagstaff,Booker-A000946}.  Later, the first
author \cite{Booker} showed that the second sequence omits infinitely
many primes, and Pollack and Trevi\~{n}o \cite{Pollack-Trevino} provided
a proof relying only on elementary number theory.
Mullin's second question remains open, but the first author showed
in \cite[Theorem~2]{Booker} that if $\{P_n\}_{n=1}^\infty$ is not
recursive then it must have Dirichlet density zero in the primes.

As for the first sequence, both of Mullin's questions remain open, and
very little is known. Shanks \cite{Shanks} conjectured that the first
sequence does indeed contain every prime, as otherwise the cumulative
products $p_1\cdots p_n$ would always be invertible modulo the smallest
omitted prime $q$, but would conspire to avoid the residue class
$-1\pmod{q}$ for all but finitely many $n$.  This observation formed
the inspiration for this paper, as we now consider what happens should
we only generate primes in a fixed residue class.

A general framework for studying Euclidean-style proofs of
the infinitude of primes in residue classes was described by Pollack
\cite{Pollack}, generalising earlier work of Murty \cite{Murty} (see
also \cite{Murty-Thain}).  Given a coprime progression $a\bmod m$,
Pollack calls a polynomial $f\in\Z[x]$ an $E'(a,m)$ polynomial if
\begin{itemize}
\item[(i)]$f(n)$ has a prime divisor $p\equiv a\pmod{m}$
for every sufficiently large $n\in\N$;
\item[(ii)]for every $p\equiv a\pmod{m}$, there exists
$n\in\Z$ such that $p\nmid f(n)$.
\end{itemize}
(We will also write $f\in E'(a,m)$ to mean the same.) Pollack proceeded
to show that any $E'(a,m)$ polynomial yields a Euclidean-style
proof of infinitude of primes $p\equiv a\pmod{m}$. From work of
Schur \cite{Schur}, we know that $E'(a,m)$ polynomials exist when
$a^2\equiv1\pmod{m}$, and Pollack proved that the converse holds assuming
Schinzel's Hypothesis H.

Pollack's proof of infinitude of primes $p\equiv a\pmod{m}$ given $f\in
E'(a,m)$ involves an arbitrary choice at each step of the construction,
and is therefore not well suited to our goal of producing explicit,
uniquely-defined sequences.  For that purpose, we make yet another
definition:
\begin{definition}\label{def:GEM}
We say that $f\in\Z[x]$ is a Generalised Euclid--Mullin
polynomial for the progression $a\bmod m$,
and write $f\in\GEM(a,m)$, if
\begin{itemize}
\item[(i)] $f(1)\ne0$ and
\item[(ii)] for every $n\in\N$ satisfying $(n,f(0))=1$ and
$n\equiv a^k\pmod{m}$ for some $k\in\N$, $f(n)$ has a prime divisor
$p$ satisfying $p\nmid f(0)$ and $p\equiv a\pmod{m}$.
\end{itemize}
\end{definition}

Given $f\in\GEM(a,m)$, we can construct a sequence of primes $(p_n)_{n\ge1}$
by choosing $p_n$ to be a prime divisor of $f(p_1\cdots p_{n-1})$
satisfying $p_n\nmid f(0)$ and $p_n\equiv a\pmod{m}$. Since $f(1)\ne0$ and
\[
f(p_1\cdots p_{n-1})\equiv f(0)\pmod{p_1\cdots p_{n-1}},
\]
it is clear that $f(p_1\cdots p_{n-1})$ never vanishes, and that the
$p_n$ are distinct, since $p_n\nmid f(0)\implies p_n\nmid p_1\cdots p_{n-1}$.

We remark that Definition~\ref{def:GEM} is not the most general one that
would enable the construction of infinite sequences of distinct primes
$p\equiv a\pmod{m}$, but it has the added feature that we can prove a
sort of equivalence between $\GEM(a,m)$ and $E'(a,m)$. More precisely,
although they are not the same as sets, they generate the same proofs,
in the sense that an $E'$ polynomial can be converted into a $\GEM$
polynomial, and vice versa:
\begin{proposition}
For any coprime progression $a\bmod m$,
\[
E'(a,m)\ne\emptyset\iff\GEM(a,m)\ne\emptyset.
\]
\end{proposition}

\begin{proof}
Let $f\in E'(a,m)$. Choose $c\in\Z$ such that $f(c)\ne0$,
and set
\[
\ell=\lcm\{p-1:p\mid f(c),\;p\equiv a\pmod{m}\}.
\]
By hypothesis, for each $p\mid f(c)$ with $p\equiv a\pmod{m}$, there
exists $n_p\in\Z$ with $p\nmid f(n_p)$. By the Chinese Remainder Theorem,
we can choose a single $A\in\N$ such that $A+c\equiv n_p\pmod{p}$
for every such $p$. Furthermore, we may take $A$ sufficiently large to
ensure that $f(n)$ is nonzero and has a prime factor $\equiv a\pmod{m}$
for every $n\ge A+c$.

Now consider $F(x)=f(Ax^\ell+c)$, and take $n\in\N$ coprime to $F(0)=f(c)$. Then
$F(n)=f(An^\ell+c)$ is nonzero and has a prime divisor $p\equiv a\pmod{m}$.
If $p\mid F(0)$ then $n^\ell\equiv1\pmod{p}$ by Fermat's Little Theorem,
and therefore $F(n)\equiv f(A+c)\equiv f(n_p)\not\equiv0\pmod{p}$.
This is a contradiction, so we must have $p\nmid F(0)$.
Since $n$ was arbitrary, $F\in\GEM(a,m)$.

Conversely, suppose $F\in\GEM(a,m)$, and set
\[
C=|F(0)|,
\quad
D=(F(1),C^\infty)=\prod_{p\mid C}p^{v_p(F(1))},
\quad
f(x)=D^{-1}F(1+CDmx).
\]
Considering the Taylor expansion of $F$ around $1$, we see that $f$
has integral coefficients.  Let $p$ be a prime number with $p\equiv
a\pmod{m}$. If $p\mid F(0)$ then by the definition of $D$, $p\nmid
D^{-1}F(1)=f(0)$.  If $p\nmid F(0)$ then choosing $n_p\in\Z$ to satisfy
$1+CDmn_p\equiv0\pmod{p}$ (which we can do since $p\nmid CDm$), we see
that $p\nmid f(n_p)$.

Now let $n\in\N$. Then by hypothesis $F(1+CDmn)$ has a prime divisor
$p\equiv a\pmod{m}$ such that $p\nmid C$. Since all prime divisors of
$D$ divide $C$, $p$ is not among them, and thus $p\mid f(n)$.
Therefore $f\in E'(a,m)$.
\end{proof}

Two natural sequences are obtained from Mullin's decree of choosing
$p_n$ as small as possible at each step, or as large as possible.
We call these the first and second generalised Euclid--Mullin sequences
associated to $f$, and denote them $\GEM_1(f;a,m)$ and $\GEM_2(f;a,m)$.
In light of Schur's result and the known and conjectured properties of
the original Euclid--Mullin sequences, we conjecture the following:
\begin{conjecture}[Generalised Euclid--Mullin conjecture]
\label{conj:GEM}
Let $a,m\in\N$ with $a^2\equiv1\pmod{m}$. Then
\begin{itemize}
\item[(i)] there exists $f\in\GEM(a,m)$ such that $\GEM_1(f;a,m)$ contains
every prime $p\equiv a\pmod{m}$;
\item[(ii)] for any $f\in\GEM(a,m)$, $\GEM_2(f;a,m)$ omits infinitely many
primes $p\equiv a\pmod{m}$.
\end{itemize}
\end{conjecture}

In this paper we make progress on the second of these
conjectures.  Our first result is a partly conditional proof of
Conjecture~\ref{conj:GEM}(ii) for $m\in\{1,2\}$. Recall that the
Extended Riemann Hypothesis (ERH) is the statement that for any number
field $L$, the Dedekind zeta function $\zeta_L(s)$ does not vanish for
$\Re(s)>\frac12$.
\begin{theorem}\label{thm:smallm}
Let $m\in\{1,2\}$ and $f\in\GEM(1,m)$. Assume either that
ERH is true or that $f$ is not of the form $c(ax+b^8)^k$ with
$c\in\Z$, $a,b,k\in\N$, $(a,b)=1$, and
$a$ odd or divisible by $32$.
Then $\GEM_2(f;1,m)$ omits infinitely many primes.
\end{theorem}

Next we study the sequences generated by the cyclotomic polynomials $\Phi_m(cx)$ for
suitable $c\in\Z$. This generalises the second author's master's thesis
\cite{Simon}, which treated the case when $m$ is a power of $2$.
A common choice in proofs of the infinitude of primes $\equiv1\pmod{m}$
is $\Phi_m(mx)$, which is a $\GEM(1,m)$ polynomial for all $m>1$. (For
$m=1$, the original Euclid--Mullin sequences are generated from
$\Phi_1(-x)$.)  In fact for any $c$ divisible by the largest prime factor
of $m$, every prime $p$ dividing $\Phi_m(cn)$ for some $n$ satisfies
$p\equiv1\pmod{m}$. However, there are many more possible choices,
as the following proposition shows.
\begin{proposition}
Let $m\in\N$ and $c\in\Z$. Then $\Phi_m(cx)$ is a $\GEM(1,m)$ polynomial
if and only if $c\in S(m)$, where
\[
S(m) \coloneq \begin{cases}
\Z \setminus \{0,1,2\}     & \text{if }m = 1,\\
2\N \cup \bigl\{c\in-\N:c\nmid 2^k+1\text{ for all }k\ge0\bigr\} & \text{if }m = 2,\\
\Z \setminus \{-2,-1,0,1\} & \text{if }m = 3,\\
\Z \setminus \{-1,0,1,2\}  & \text{if }m = 6,\\
\Z \setminus \{-1,0,1\}    & \text{otherwise}.
\end{cases}
\]
\end{proposition}
\begin{proof}
Combining Zsigmondy's theorem \cite{Zsigmondy} with the identities
\begin{align*}
m \equiv 1 \pmod2&\implies \Phi_m(-x) = \pm\Phi_{2m}(x),\\
m \equiv 2 \pmod4&\implies \Phi_m(-x) = \pm\Phi_{m/2}(x),\\
m \equiv 0 \pmod4&\implies \Phi_m(-x) = \Phi_m(x),
\end{align*}
we see that for any $x\in\Z\setminus\{-1,0,1\}$,
$\Phi_m(x)$ is nonzero and has a prime divisor $p\equiv1\pmod{m}$ except
in the following cases:
\begin{itemize}
\item $m=1$ and $x=2$;
\item $m=2$ and $|x+1|$ is a power of $2$;
\item $m=3$ and $x=-2$;
\item $m=6$ and $x=2$.
\end{itemize}
One can also see that the values of $\Phi_m(x)$ for $x\in\{-1,0,1\}$
are never suitable except when $(m,x)=(1,-1)$.

In order for $\Phi_m(cx)$ to be $\GEM(1,m)$ polynomial, we need $cn$ to
avoid the exceptions above for every natural number $n\equiv1\pmod{m}$.
For $m\ne2$ it is straightforward to see that this is equivalent to
$c\in S(m)$.

For $m=2$ and odd $c>0$, there are infinitely many odd $n\in\N$ such
that $cn+1$ is a power of $2$ (choose $k\in\N$ divisible by the order
of $2\bmod c$ and set $n=(2^k-1)/c$). Similarly, for $m=2$ and $c<0$,
there exists odd $n\in\N$ such that $|cn+1|$ is a power of $2$ if and
only if $c\mid 2^k+1$ for some $k$.
\end{proof}

Again we can give a conditional proof of Conjecture~\ref{conj:GEM}(ii)
in this setting:
\begin{theorem}\label{thm:ERH}
Let $m\in\N$ and $c\in S(m)$. Then, assuming ERH,
$\GEM_2(\Phi_m(cx);1,m)$ omits infinitely many primes $p\equiv1\pmod{m}$.
\end{theorem}
\noindent
For $m>2$ we will see that we can reduce to
the Generalised Lindel\"of Hypothesis
(GLH) for Dirichlet $L$-functions, which is implied by ERH.
Moreover, we can remove all conditional hypotheses in some cases:
\begin{theorem}\label{thm:unconditional}
Let $m\in\{1,\ldots,10,12,14,18\}$ and $c\in S(m)$,
with the following further assumptions:
\begin{align*}
m=1&\implies c\in\N\cup(2\Z\setminus32\Z)\cup\{-5\}\cup\bigl\{1-2^k:k\in\N\bigr\},\\
m=2&\implies c\in-\N\cup(2\Z\setminus32\Z).
\end{align*}
Then $\GEM_2(\Phi_m(cx);1,m)$ omits infinitely many primes $p\equiv1\pmod{m}$.
\end{theorem}
\noindent
(The restrictions on $c$ for $m\in\{1,2\}$
arise because $\Phi_m(cx)$ has degree $1$;
in view of Theorem~\ref{thm:smallm},
that is the only obstruction to giving an unconditional
proof of the conjecture in those cases.)

Finally, we show unconditionally that at least one prime
is omitted for infinitely many $m$:
\begin{theorem}\label{thm:firstprime}
Let $m>2$ be a power of $2$, and let $c\in S(m)$.
Then $\GEM_2(\Phi_m(cx);1,m)$ omits the smallest prime $p\equiv1\pmod{m}$
if $(m,c)\notin\{(4,\pm2),(4,\pm3),(8,\pm2)\}$,
and the second smallest such prime otherwise.
\end{theorem}

\subsection{Outline}
The contents of the paper are as follows:
\begin{itemize}
\item In Section~\ref{sec:unconditional} we develop the character sum
machinery necessary for Theorem~\ref{thm:unconditional},
culminating in the proof of Theorem~\ref{thm:quadchars}.
This is modelled on the proof of \cite[Lemma~4]{Booker}, and
the main ingredient is again Burgess' estimate \cite{Burgess} for short
character sums, leveraged by the Dirichlet hyperbola method.
\item Section~\ref{sec:ERH} derives a similar result for higher-order
characters, Theorem~\ref{thm:orderl}, to be used in the proof of
Theorems~\ref{thm:smallm} and \ref{thm:ERH}. Since progress toward GLH for Dedekind zeta
functions of non-abelian extensions of $\Q$ is much more limited than for
abelian extensions, we derive this conditionally on ERH.
\item Section~\ref{sec:diophantine} derives or quotes a few
miscellaneous results from Diophantine equations that are
needed in the proofs of Theorems~\ref{thm:smallm}--\ref{thm:firstprime}.
\item Finally, we assemble these ingredients in Section~\ref{sec:proof}
to complete the proofs.
\end{itemize}

\subsection*{Acknowledgements}
The first author is grateful to Dan Fretwell for suggesting
the study of cyclotomic analogues of the Euclid--Mullin sequences.

\section{Unconditional analytic estimates}\label{sec:unconditional}
The goal of this section is to prove an analogue of \cite[Lemma~4]{Booker}
adapted to the setting of cyclotomic polynomials; see
Theorem~\ref{thm:quadchars}.
Fix $m\in\N$, let $K=\Q(e^{2\pi i/m})$ be the $m$th cyclotomic field,
and denote its ring of integers by $\OK=\Z[e^{2\pi i/m}]$. Let
$\zeta_K(s)$ denote the Dedekind zeta function of $K$, so that
\[ \zeta_K(s)=\prod_\pp\frac1{1-N(\pp)^{-s}}, \]
where the product runs over prime ideals $\pp\subset\OK$.
Define coefficients $a(n)$ by
\[
\zeta_K(s)=\sum_{n=1}^\infty\frac{a(n)}{n^s}.
\]
Further set
$T=\{\pp:\pp\mid m\OK\text{ and }f(\pp)=1\}\cup\{\pp:f(\pp)\text{ is odd and} > 1\}$, where
$f(\pp)$ denotes the inertia degree of $\pp$,
and define $a^\flat(n)$ such that
\[
\zeta_K(s)\prod_{\pp\in T}(1-N(\pp)^{-s})=\sum_{n=1}^\infty\frac{a^\flat(n)}{n^s}.
\]

\begin{lemma}\label{lem:burgess}
Let $\chi$ be a character of modulus $q$ and conductor
$q^\ast>1$. Let $q_1$ be the maximal cubefree unitary divisor of $q^\ast$, and
set $q_2=q^\ast/q_1$, $q_0=\prod_{\substack{p\mid q\\p\nmid q^\ast}}p$.
Then
\[
\sum_{n=M+1}^{M+N}\chi(n)\ll_{r,\varepsilon}
q_0^\varepsilon q_1^{\frac{r+1}{4r^2}+\varepsilon}q_2^{\frac1r}N^{1-\frac1r}
\quad\text{for all }M,N,r\in\N\text{ and }\varepsilon>0.
\]
\end{lemma}
\begin{proof}
When $\chi$ is primitive and $q_2=1$, this is the result of Burgess
\cite[Theorem~2]{Burgess}. As explained in \cite[(12.56)]{Iwaniec-Kowalski}, we can
extend to non-principal primitive characters by breaking into progressions modulo
$q_2$ and applying Burgess' result. Finally, we can extend to all
non-principal characters by sieving:
\[
\sum_{n=M+1}^{M+N}\chi(n)
=\sum_{n=M+1}^{M+N}\sum_{d\mid(n,q_0)}\mu(d)\chi^\ast(n)
=\sum_{d\mid q_0}\mu(d)\chi^\ast(d)
\sum_{\frac{M}{d}<n\le\frac{M+N}{d}}\chi^\ast(n)
\ll_{r,\varepsilon}q_0^\varepsilon q_1^{\frac{r+1}{4r^2}+\varepsilon}q_2^{\frac1r}N^{1-\frac1r}.
\]
\end{proof}

\begin{lemma}\label{lem:GLH}
Let $\chi$ be a non-trivial character of modulus $q$.
Assuming GLH for Dirichlet $L$-functions, we have
\[
\sum_{n\le x}\chi(n)\ll_\varepsilon q^\varepsilon\sqrt{x}.
\]
\end{lemma}
\begin{proof}
By periodicity we may assume without loss of generality that $x\le q$.
By Perron's formula \cite[Corollary~2.1]{Tenenbaum}, for $x\ge2$,
\[
\sum_{n\le x}\chi(n)
=\frac1{2\pi i}\int_{1+\frac1{\log{x}}-ix}^{1+\frac1{\log{x}}+ix}
L(s,\chi)x^s\,\frac{ds}{s}
+O(\log{x}).
\]
In turn, under GLH we have $L(s,\chi)\ll_\varepsilon|qs|^{\frac{\varepsilon}{2}}$
uniformly for $\Re(s)\ge\frac12$. Shifting the contour to $\Re(s)=\frac12$,
we see that
\[
\sum_{n\le x}\chi(n)\ll_\varepsilon(qx)^{\frac{\varepsilon}{2}}\sqrt{x}
\le q^\varepsilon\sqrt{x}.
\]
\end{proof}

\begin{lemma}\label{lem:generalburgess}
Let $\chi$ be a quadratic character modulo $q$,
not necessarily primitive, with conductor not dividing $m$.
For $r\in\N$, define
\[
\nu(r)=\begin{cases}
0&\text{if $r=2$ and GLH holds for Dirichlet $L$-functions},\\
\frac{r+1}{4r^2}&\text{otherwise}.
\end{cases}
\]
Then for any $r\in\N$ and $\varepsilon>0$,
\[
\sum_{n\le x}a^\flat(n)\chi(n)\ll_{m,r,\varepsilon} q^{\nu(r)+\varepsilon}x^{1-\frac1{\varphi(m)r}}.
\]
\end{lemma}

\begin{proof}
We first show that it suffices to prove the corresponding estimate
with $a^\flat(n)$ replaced by $a(n)$. To that end, define coefficients
$b(n)$ such that
\[
\prod_{\pp\in T}(1-N(\pp)^{-s})=\sum_{n=1}^\infty\frac{b(n)}{n^s},
\]
and note that
\[
\sum_{n=1}^\infty\frac{|b(n)|}{n^\sigma}<\infty
\quad\text{for all }\sigma>\tfrac13.
\]
Then $a^\flat=a\ast b$, so
assuming the desired result holds for $a(n)$, we have
\begin{align*}
\sum_{n\le x}a^\flat(n)\chi(n)
&=\sum_{n\le x}\sum_{d\mid n}b(d)a(n/d)\chi(n)
=\sum_{d\le x}b(d)\chi(d)\sum_{n\le\frac{x}{d}}a(n)\chi(n)\\
&\ll_{m,r,\varepsilon} \sum_{d\le x}|b(d)\chi(d)|q^{\nu(r)+\varepsilon}
\left(\frac{x}{d}\right)^{1-\frac1{\varphi(m)r}}
\le q^{\nu(r)+\varepsilon}x^{1-\frac1{\varphi(m)r}}
\sum_{d=1}^\infty\frac{|b(d)|}{d^{1-\frac1{\varphi(m)r}}}.
\end{align*}
The final series over $d$ always converges, since
$1-\frac1{\varphi(m)r}>\frac13$ unless $m\le 2$, and in those cases
there are no primes $\pp$ with $f(\pp)>1$, so $b(d)=0$ for $d>2$.

Let $\xi_1,\ldots,\xi_{\varphi(m)}$ denote the primitive characters of
conductor dividing $m$, and put $\chi_i=\xi_i\chi$, so that
$a(n)\chi(n) = (\chi_1\ast\cdots\ast\chi_{\varphi(m)})(n)$.
We will prove by induction on $k\ge1$ that
\begin{equation}\label{eq:chik}
\sum_{n\le x}(\chi_1\ast\cdots\ast\chi_k)(n)\ll_{k,m,r,\varepsilon} q^{\nu(r)+\varepsilon}x^{1-\frac1{kr}}.
\end{equation}
The result then follows on taking $k=\varphi(m)$.

For a single character $\chi_i=\xi_i\chi$, first note that since $\chi$
is quadratic, its conductor is cubefree apart from a possible factor of
$8$. Thus, if $q_0,q_1,q_2$ are as in Lemma~\ref{lem:burgess} applied
to $\chi_i$ then $q_2\mid 8m$. Hence, by Lemmas~\ref{lem:burgess}
and \ref{lem:GLH}, we have
\[
\sum_{n\le x}\chi_i(n)\ll_{m,r,\varepsilon}
q^{\nu(r)+\varepsilon}x^{1-\frac1r},
\]
which implies the $k=1$ case of \eqref{eq:chik}.

Suppose \eqref{eq:chik} holds for some $k<\varphi(m)$, and
write
\[
f=\chi_1\ast\cdots\ast\chi_k,
\quad g=\chi_{k+1},
\quad F(x)=\sum_{n\le x}f(n),
\quad G(x)=\sum_{n\le x}g(n).
\]
Then by the Dirichlet hyperbola method, for any $y>0$, we have
\[
\sum_{n\le x}(f\ast g)(n)
=\sum_{n\le y}F\!\left(\frac{x}{n}\right)g(n)
+\sum_{n\le\frac{x}{y}}f(n)G\!\left(\frac{x}{b}\right)
-F\!\left(\frac{x}{y}\right)G(y).
\]
Taking $y=x^{\frac1{k+1}}$ and using
the estimates
\[
F(x)\ll_{k,m,r,\varepsilon} q^{\nu(r)+\varepsilon}x^{1-\frac1{kr}},
\quad G(x)\ll_{m,r,\varepsilon} q^{\nu(r)+\varepsilon}x^{1-\frac1r},
\quad |f(n)|\le d_k(n)\ll_{k,\varepsilon} n^\varepsilon,
\quad |g(n)|\le1,
\]
we obtain
\begin{equation}\label{eq:k+1}
\sum_{n\le x}(f\ast g)(n)
\ll_{k,m,r,\varepsilon} q^{\nu(r)+\varepsilon}x^{1-\frac1{(k+1)r}}
+ q^{2\nu(r)+\varepsilon}x^{1-\frac2{(k+1)r}}.
\end{equation}
In view of the trivial bound
\[
\sum_{n\le x}(f\ast g)(n)\ll_{k,\varepsilon}x^{1+\varepsilon},
\]
we may assume without loss of generality that
$x\ge q^{(k+1)r\nu(r)}$. Hence
the first term on the right-hand side of \eqref{eq:k+1}
majorises the second, so we get
\[
\sum_{n\le x}(\chi_1\ast\cdots\ast\chi_{k+1})(n)
\ll_{k,m,r,\varepsilon} q^{\nu(r)+\varepsilon}x^{1-\frac1{(k+1)r}}.
\]
This completes the proof.
\end{proof}

\begin{proposition}\label{prop:vin}
Let $\chi$ be a quadratic character modulo $q$,
not necessarily primitive, with conductor not dividing $m$.
Then there is a prime
$p\equiv1\pmod{m}$ such that $\chi(p)=-1$ and
\[
p\ll_{m,\varepsilon}
\begin{cases}
q^\varepsilon&\text{if GLH holds for Dirichlet $L$-functions},\\
q^{\frac{\varphi(m)}{4\sqrt{e}}+\varepsilon}&\text{otherwise}.
\end{cases}
\]
\end{proposition}

\begin{proof}
Note that if $a^\flat(n)\ne0$ then for every prime $p\mid n$ we have either
$p\equiv1\pmod{m}$ or $2\mid v_p(n)$. Hence if $a^\flat(n)\chi(n)<0$ then
there exists $p\mid n$ such that $p\equiv1\pmod{m}$ and $\chi(p)=-1$.

Fix a small $\delta>0$, and suppose there is no such prime $\le y=q^{\delta+\alpha e^{\delta/2}}$, where
\[
\alpha\coloneq\begin{cases}
0&\text{if GLH holds for Dirichlet $L$-functions},\\
\frac{\varphi(m)}{4\sqrt{e}}&\text{otherwise}.
\end{cases}
\]
Then $a^\flat(n)(1-\chi(n))=0$
whenever $n$ is coprime to $q$ and $y$-smooth.
Considering $x\in\bigl[y^{1+\delta},y^2\bigr]$, we have
\begin{equation}\label{eq:vin}
\begin{aligned}
\sum_{n\le x}a^\flat(n)\chi(n)
&=\sum_{\substack{n\le x\\(n,q)=1}}a^\flat(n)
-\sum_{\substack{n\le x\\(n,q)=1}}a^\flat(n)(1-\chi(n))\\
&=\sum_{\substack{n\le x\\(n,q)=1}}a^\flat(n)
-\sum_{\substack{y<p\le x\\p\nmid q}}a^\flat(p)
\sum_{\substack{n\le x/p\\(n,q)=1}} a^\flat(n)(1-\chi(pn))\\
&\ge\sum_{\substack{n\le x\\(n,q)=1}}a^\flat(n)
-2\sum_{y<p\le x}a(p)
\sum_{\substack{n\le x/p\\(n,q)=1}} a^\flat(n).
\end{aligned}
\end{equation}

By a classical result going back to Weber \cite[\S194]{Weber}, we have
\[
\sum_{n\le x}a(n)=\kappa x+O\bigl(x^\theta\bigr)
\quad\text{for all }x>0,
\]
where $\kappa>0$ and $\theta<1$ depend only on $m$.
We may assume without loss of generality that $\theta>\frac13$. Then
applying a sieve as in the proof of Lemma~\ref{lem:generalburgess},
we derive
\begin{equation}\label{eq:aflatasymptotic}
\sum_{n\le x}a^\flat(n)=\kappa^\flat x+O\bigl(x^\theta\bigr)
\quad\text{for all }x>0,
\end{equation}
where $\kappa^\flat=\kappa\prod_{\pp\in T}(1-N(\pp)^{-1})>0$.

Similarly, to sieve out primes dividing $q$, let $b_q(n)$ be the coefficients
satisfying
\[
\prod_{\substack{\pp\mid q\OK\\\pp\notin T}}(1-N(\pp)^{-s})
=\sum_{n=1}^\infty\frac{b_q(n)}{n^s}.
\]
Then
\begin{align*}
\sum_{\substack{n\le x\\(n,q)=1}}a^\flat(n)
&=\sum_{n\le x}\sum_{d\mid n}b_q(d) a^\flat(n/d)
=\sum_{d=1}^\infty b_q(d)\sum_{n\le x/d}a^\flat(n)
=\sum_{d=1}^\infty b_q(d)\left(\frac{\kappa^\flat x}{d}
+O\bigl((x/d)^\theta\bigr)\right)\\
&=\kappa^\flat A(q)x +O\bigl(B(q)x^\theta\bigr),
\end{align*}
where
\[
A(q)=\sum_{d=1}^\infty\frac{b_q(d)}{d}
=\prod_{\substack{\pp\mid q\OK\\\pp\notin T}}(1-N(\pp)^{-1})
\ge\left(\frac{\varphi(q)}{q}\right)^{\varphi(m)}
\gg_m(\log\log{q})^{-\varphi(m)}
\]
and
\[
B(q)=\sum_{d=1}^\infty\frac{|b_q(d)|}{d^\theta}
=\prod_{\substack{\pp\mid q\OK\\\pp\notin T}}\bigl(1+N(\pp)^{-\theta}\bigr)
\ll_{m,\varepsilon} q^\varepsilon.
\]

Returning to \eqref{eq:vin}, we split the $p$ sum
over the ranges $y<p\le z$ and $z<p\le x$, where
$z=(\frac{A(q)}{B(q)})^{\frac1{1-\theta}}x\gg_\varepsilon x^{1-\varepsilon}$.
(Note that $z>y$ when $q$ is sufficiently large.)
For the small primes, we have
\begin{align*}
2\sum_{y<p\le z}a(p)\sum_{\substack{n\le x/p\\(n,q)=1}}a^\flat(n)
&=2\sum_{y<p\le z}a(p)\left(\frac{\kappa^\flat A(q)x}{p}+
O\bigl(B(q)(x/p)^\theta\bigr)\right)\\
&=2\kappa^\flat A(q)x\sum_{y<p\le z}\frac{a(p)}{p}
+O\!\left(B(q)x^\theta\sum_{y<p\le z}|a(p)|p^{-\theta}\right).
\end{align*}
Recall that $a(p)=\sum_\xi\xi(p)$,
where $\xi$ ranges over the primitive characters of conductor dividing $m$.
By the prime number theorem in arithmetic progressions, we have
\[
\sum_{y<p\le z}\frac{a(p)}{p}=\log\!\left(\frac{\log{z}}{\log{y}}\right)
+O_m\bigl((\log{y})^{-1}\bigr),
\]
and using that $|a(p)|\le\varphi(m)$, we have
\[
\frac{B(q)}{A(q)}\sum_{y<p\le z}|a(p)|(x/p)^\theta
\ll_m\frac{B(q)}{A(q)}x^\theta\frac{z^{1-\theta}}{\log{x}}
=\frac{x}{\log{x}}.
\]
Thus the small prime sum is
\[
2\kappa^\flat A(q)x\left(\log\!\left(\frac{\log{z}}{\log{y}}\right)
+O_m\bigl((\log{x})^{-1}\bigr)\right).
\]

For the large primes we change order of summation:
\begin{align*}
2\sum_{z<p\le x}a(p)&\sum_{\substack{n\le x/p\\(n,q)=1}}a^\flat(n)
\le 2\sum_{\substack{n\le x/z\\(n,q)=1}}a^\flat(n)\sum_{p\le x/n}a(p).
\end{align*}
Applying the prime number theorem in arithmetic progressions, the inner sum is
\[
\sum_{p\le x/n}a(p)=\frac{x}{n\log{x}}\left(1
+O_m\!\left(\frac{1+\log(x/z)}{\log{x}}\right)\right),
\]
so we get
\begin{align*}
\frac{2x}{\log{x}}\left(1+O\!\left(\frac{1+\log(x/z)}{\log{x}}\right)\right)
\sum_{\substack{n\le x/z\\(n,q)=1}}\frac{a^\flat(n)}{n}.
\end{align*}
Now, applying partial summation to \eqref{eq:aflatasymptotic}, we have
\[
\sum_{n\le x}\frac{a^\flat(n)}{n}
=\kappa^\flat\log{x}+c+O\bigl(x^{\theta-1}\bigr)
\quad\text{for all }x>0,
\]
for a certain constant $c$ depending only on $m$, and hence
\begin{align*}
\sum_{\substack{n\le x\\(n,q)=1}}\frac{a^\flat(n)}{n}
&=\sum_{d=1}^\infty\frac{b_q(d)}{d}
\sum_{n\le x/d}\frac{a^\flat(n)}{n}
=\sum_{d=1}^\infty\frac{b_q(d)}{d}
\bigl(\kappa^\flat\log(x/d)+c+O\bigl((x/d)^{\theta-1}\bigr)\bigr)\\
&=A(q)(\kappa^\flat\log{x}+c)
-\kappa^\flat\sum_{d=1}^\infty\frac{b_q(d)\log{d}}{d}
+O\bigl(B(q)x^{\theta-1}\bigr)\\
&=A(q)\left(\kappa^\flat\log{x}+c
+\kappa^\flat\sum_{\substack{\pp\mid q\OK\\\pp\notin T}}\frac{\log{N(\pp)}}{N(\pp)-1}\right)
+O\bigl(B(q)x^{\theta-1}\bigr)\\
&=\kappa^\flat A(q)\bigl(\log{x}+O(\log\log{q})\bigr)
+O\bigl(B(q)x^{\theta-1}\bigr).
\end{align*}
Replacing $x$ by $x/z=(B(q)/A(q))^{\frac1{1-\theta}}$, we see that
the large prime sum is at most
\begin{align*}
&2\kappa^\flat A(q)\frac{x}{\log{x}}
\left(1+O\!\left(\frac{1+\log(x/z)}{\log{x}}\right)\right)
(\log(x/z) + O(\log\log{q}))\\
&=2\kappa^\flat A(q)x
\left(\frac{\log(x/z)}{\log{x}}
+O\!\left(\frac{\log\log{q}}{\log{q}}\right)\right).
\end{align*}

Combining with the small prime sum, we get
\begin{align*}
&\le 2\kappa^\flat A(q)x\left(\log\!\left(\frac{\log{x}}{\log{y}}\right)
+\log\!\left(\frac{\log{z}}{\log{x}}\right)
+\frac{\log(x/z)}{\log{x}}
+O\!\left(\frac{\log\log{q}}{\log{q}}\right)\right)\\
&\le 2\kappa^\flat A(q)x\left(\log\!\left(\frac{\log{x}}{\log{y}}\right)
+O\!\left(\frac{\log\log{q}}{\log{q}}\right)\right),
\end{align*}
and so altogether we have
\[
\sum_{n\le x}a^\flat(n)\chi(n)
\ge \kappa^\flat A(q) x
\left(1-2\log\!\left(\frac{\log{x}}{\log{y}}\right)
+O\!\left(\frac{\log\log{q}}{\log{q}}\right)\right).
\]
Now set $x=y^{e^{\frac12(1-\delta)}}$
. Then
\[
1-2\log\!\left(\frac{\log{x}}{\log{y}}\right)
+O\!\left(\frac{\log\log{q}}{\log{q}}\right)
=\delta+O\!\left(\frac{\log\log{q}}{\log{q}}\right).
\]
This is at least $\delta/2$ for sufficiently large $q$, so applying
Lemma~\ref{lem:generalburgess} with $r=2$ when GLH holds
and $r>\frac{\varphi(m)}{4\delta\exp((1-\delta)/2)}$ otherwise,
we have
\begin{align*}
(\log\log{q})^{-\varphi(m)}&\ll
A(q)\ll_\delta\frac1x\left|\sum_{n\le x}a^\flat(n)\chi(n)\right|
\ll_{r,\varepsilon} q^{\nu(r)+\varepsilon}
x^{-\frac1{\varphi(m)r}}\\
&=q^{\varepsilon-\frac{\delta\exp((1-\delta)/2)}{\varphi(m)r}}
\begin{cases}
1&\text{if GLH holds},\\
q^{\frac1{4r^2}}&\text{otherwise}.
\end{cases}
\end{align*}
Choosing $\varepsilon$ sufficiently small,
we obtain a contradiction for sufficiently large $q$.

Hence for $q\gg_{m,\delta}1$ there exists a prime
$p\le y=q^{\delta+\alpha e^{\delta/2}}$ satisfying
$p\equiv1\pmod{m}$ and $\chi(p)=-1$.
By Dirichlet's theorem, there is always a suitable
$p\ll_q 1$, so we can choose an implied constant to cover
the small values of $q$. Since $\delta$ was arbitrary, we thus have
$p\ll_{m,\varepsilon}q^{\alpha+\varepsilon}$.
\end{proof}

\begin{theorem}\label{thm:quadchars}
For $i=1,\ldots,r$, let $\chi_i\pmod{q_i}$ be a quadratic character,
and let $\epsilon_i\in\{\pm1\}$.
Assume that no non-empty product of $\chi_i$s has conductor
dividing $m$.
Then there is a squarefree positive integer $n$ with at most
$r$ prime factors $p_j$, each satisfying
\[
p_j\equiv1\!\!\!\!\pmod{m}\quad\text{and}\quad
p_j\ll_{m,\varepsilon}\begin{cases}
(q_1\cdots q_r)^\varepsilon&\text{if GLH holds for Dirichlet $L$-functions},\\
(q_1\cdots q_r)^{\frac{\varphi(m)}{4\sqrt{e}}+\varepsilon}&\text{otherwise},
\end{cases}
\]
such that $\chi_i(n)=\epsilon_i$ for all $i=1,\ldots,r$.
\end{theorem}
\begin{proof}
This generalises \cite[Lemma~4]{Booker}, and the proof is identical,
substituting Proposition~\ref{prop:vin} in place of \cite[Lemma~3]{Booker}.
\end{proof}

\section{Conditional analytic estimates}\label{sec:ERH}
Next we derive an analogue of Theorem~\ref{thm:quadchars}
for higher-order characters that will be useful
in the proof of Theorem~\ref{thm:smallm}.
\begin{lemma}\label{lem:Kummer}
Let $\ell$ be an odd prime number, and let $q>1$ be an integer coprime
to $\ell$ and not divisible by the $\ell$th power of any prime. Assuming ERH
for the Dedekind zeta function of the Kummer extension
$\Q(\sqrt[\ell]{q},e^{2\pi i/\ell})$, there exists a prime
$P\ll\ell^4\log^2\left(\ell\prod_{p\mid q}p\right)$
such that $q$ is not an $\ell$th power modulo $P$.
\end{lemma}
\begin{proof}
Let $K=\Q(e^{2\pi i/\ell})$ and $L=K(\sqrt[\ell]{q})$.
By \cite[Lemma~5]{Booker}, $L/K$ is cyclic of degree $\ell$,
and a rational prime $P\nmid\ell q$ splits completely in $L$ if and only if
$P\equiv1\pmod{\ell}$ and $q$ is an $\ell$th power modulo $P$.

Hence it suffices to find $P\nmid q$ with $P\equiv1\pmod{\ell}$ such that
$P$ does not split completely in $L$. Equivalently, we seek an unramified
degree $1$ prime $\pp$ of $K$ such that the Artin symbol
$\left(\frac{L/K}{\pp}\right)$ is non-trivial.
By \cite[Theorem~3.1(2)]{Bach-Sorenson}, there is such a $\pp$ satisfying
\[
N_{K/\Q}\pp\le(1+o(1))\log^2\Delta_L,
\]
where $\Delta_L$ is the absolute discriminant of $L/\Q$, and
$o(1)$ is a quantity bounded by an absolute constant and
tending to $0$ as $\Delta_L\to\infty$.

Since $\ell\nmid q$ and $q$ is free of $\ell$th powers, local Kummer theory shows
that $L$ has absolute discriminant
\[
\Delta_L=\Delta_K^{[L:K]}N_{K/\Q}(\Disc(L/K))
=\ell^{\ell(\ell-2)}\prod_{p\mid q}p^{(\ell-1)^2},
\]
and thus we can take
\[
P = N_{K/\Q}\pp\ll\log^2\Delta_L\ll\ell^4\log^2\left(\ell\prod_{p\mid q}p\right).
\]
\end{proof}

\begin{theorem}\label{thm:orderl}
Let $\ell$ be an odd prime number, let $q_1,\ldots,q_r$ be distinct
primes different from $\ell$, and let $\epsilon_1,\ldots,\epsilon_r\in\C^\times$ be arbitrary $\ell$th roots of unity. Then, assuming ERH,
there exists a primitive character $\chi$ of squarefree conductor $d$
with the following properties:
\begin{itemize}
\item $\chi$ has order dividing $\ell$;
\item every prime $p\mid d$ satisfies $p\ll\ell^4\log^2(\ell q_1\cdots q_r)$;
\item $\chi(q_j)=\epsilon_j$ for $j=1,\ldots,r$.
\end{itemize}
\end{theorem}

\begin{proof}
Fix a primitive $\ell$th root of unity $\zeta\in\C^\times$ and
a large real number $x$. Let $\{p_1,\ldots,p_n\}$
be the set of primes $p\le x$ satisfying $p\equiv1\pmod{\ell}$ and
$p\nmid q_1\cdots q_r$. For each $i=1,\ldots,n$, let $\chi_i$ be a
character mod $p_i$ of order $\ell$, and let
$w_i=(w_{i1},\ldots,w_{ir})\in\F_\ell^r$
be the vector such that $\chi_i(q_j)=\zeta^{w_{ij}}$.
Suppose that the $w_i$ span a proper subspace of $\F_\ell^r$.
Then there is a non-zero vector $v\in\F_\ell^r$ such that
$v\cdot w_i=0$ for every $i$.
Writing $v=(v_1,\ldots,v_r)\bmod\ell$, where $v_j\in[0,\ell)\cap\Z$,
and $q=q_1^{v_1}\cdots q_r^{v_r}$,
this implies that $\chi_i(q)=1$ for every $i$,
which in turn implies that $q$ is an $\ell$th power
modulo $p_i$. Since $v$ is non-zero, $q$ satisfies the hypotheses
of Lemma~\ref{lem:Kummer}, and this results in a contradiction
when $x\gg\ell^4\log^2(\ell q_1\cdots q_r)$.

Therefore the $w_i$ span $\F_\ell^r$, so there are indices
$i_1,\ldots,i_r\le n$ such that $\{w_{i_1},\ldots,w_{i_r}\}$ is
a basis. Let $b=(b_1,\ldots,b_r)\in\F_\ell^r$ be the vector such that
$\epsilon_j=\zeta^{b_j}$, and write $b=a_1w_{i_1}+\cdots+a_rw_{i_r}$.
Then the character
\[
\chi=\prod_{\substack{1\le k\le r\\a_k\ne0}}\chi_{i_k}^{a_k}
\]
fulfills the requirements of the lemma.
\end{proof}

\section{Results on Diophantine equations}\label{sec:diophantine}
\begin{lemma}\label{lem:siegel}
Let $f\in\Z[x]$ be a squarefree polynomial of degree at least $3$.
Then $f(n)$ is a square for at most finitely many $n\in\Z$.
\end{lemma}
\begin{proof}
This follows from Siegel's theorem on integral points \cite{Siegel}.
\end{proof}
We also provide a simple proof of this result for polynomials of even degree
with square leading coefficient:
\begin{lemma}\label{lem:squarevalues}
Suppose $f\in\Z[x]$ has even degree and leading coefficient $a^2$
for some $a\in\Z$. Then either $f=g^2$ for some $g\in\Z[x]$ or
$f(n)$ is a square for at most finitely many $n\in\Z$.
\end{lemma}

\begin{proof}
Let $d=\deg{f}$. The conclusion is clear if $f$ is constant, so we may
assume that $d>0$.  We claim that there are polynomials $s,r\in\Z[x]$
with $\deg{r}<\frac{d}{2}$ such that
\[
(2a)^{2d-2}f(x)=s(x)^2+r(x).
\]
To see this, write $f(x)=\sum_{i=0}^d f_i x^{d-i}$ and
$s(x)=\sum_{i=0}^{\frac{d}{2}}s_ix^{\frac{d}{2}-i}$, where we define
$s_0=2^{d-1}a^d$, and recursively for $i=1,2,\ldots,\frac{d}{2}$,
\[
s_i = (2a)^{d-2}f_i - (2a)^{-d}\sum_{j=1}^{i-1}s_js_{i-j}.
\]
Then it is clear that $r(x)\coloneq(2a)^{2d-2}f(x)-s(x)^2$ has degree less
than $\frac{d}{2}$, and by induction we see that $s_i\in(2a)^{d-2i}\Z$
for every $i>0$, so $s,r\in\Z[x]$.

Since $s$ has degree $\frac{d}{2}$, for large $x$ we have $|r(x)-1|<2|s(x)|$,
and it follows that $(2a)^{2d-2}f(x)$ is strictly between
$(|s(x)|-1)^2$ and $(|s(x)|+1)^2$.

If $r$ is not identically $0$
then it has at most $\frac{d}{2}-1$ zeros, and therefore $f(x)$ has at
most finitely many square values.
Otherwise, if $r$ is identically $0$ then Gauss' Lemma implies that $f=g^2$ for some $g\in\Z[x]$.
\end{proof}

\begin{lemma}\label{lem:nonsquare}
Let $N\in\N$ with $\sqrt{N}\notin\N$. Then there are infinitely many primes $p$ such that the Kronecker symbol $\left(\frac{p}{N}\right)$ equals $-1$.
\end{lemma}

\begin{proof}
Since $N$ is not a square,
$\left(\frac{\cdot}{4N}\right)$ is a non-principal quadratic character modulo $4N$,
and hence we have $\sum_{n=1}^{4N}\left(\frac{n}{4N}\right) = 0$.
Since $\left(\frac{1}{4N}\right)=1$, it follows that there is some $n\le4N$ with
$\left(\frac{n}{4N}\right)=-1$.

We conclude with a version of Euclid's argument. Let $p_1,\ldots,p_k$ be a (possibly empty)
sequence of distinct primes satisfying $\left(\frac{p_i}{4N}\right)=-1$ and $p_i\nmid n$,
and consider the number $t=4Np_1\cdots p_k+n$. Since $t\equiv n\pmod{4N}$, we have
$\left(\frac{t}{4N}\right)=-1$, and by multiplicativity it follows that there is a prime
$p\mid t$ with $\left(\frac{p}{4N}\right)=-1$. By hypothesis $(t,n)=(4Np_1\cdots p_k,n)=1$,
so $p\nmid n$, and it follows that $p\notin\{p_1,\ldots,p_k\}$.
\end{proof}

\begin{lemma}\label{lem:chebotarev}
Let $K$ be a real quadratic field with ring of integers $\OK$ and
fundamental unit $\epsilon\in\OK^\times$. Let $\alpha_1,\ldots,\alpha_n\in
K^\times\setminus\OK^\times$, let $\ell>2$ be a prime such that
the fractional ideals $\alpha_i\OK$ are not $\ell$th powers,
and let $m$ be an odd positive integer.
Then there are infinitely many primes $\pp\subset\OK$ such that
the order of $\epsilon\bmod\pp$ is odd and divisible
by $m$, and $\alpha_1,\ldots,\alpha_n$ are $\ell$th power
non-residues modulo $\pp$.
\end{lemma}
\begin{proof}
We may assume without loss of generality that $\ell\mid m$.
For any $r\in\N$, let $\zeta_r$ denote a primitive $r$th root of unity,
and consider primes $\pp$ that split completely in
$L=K(\zeta_{8m},\sqrt[8]\epsilon)$, but not
in $L(\zeta_{16})$, $L(\sqrt[\ell]{\alpha_i})$,
or $L(\sqrt[p]{\epsilon})$ for any prime $p\mid m$.
For any such $\pp$ we have $N(\pp)\equiv9\pmod{16}$, and
it follows that $-1$ is not in the subgroup
of $(\OK/\pp)^\times$ generated by
$\epsilon$. Moreover, $\alpha_1,\ldots,\alpha_n$ are
$\ell$th power non-residues mod $\pp$, and
$\epsilon$ is not a $p$th power non-residue for every
$p\mid m$, so it must have order divisible by $m$.
By the Chebotarev density theorem,
the set of $\pp$ satisfying all these conditions
has relative density at least
\[
\frac1{[L:K]}\left(1-\frac1{[L(\zeta_{16}):L]}\right)
\prod_{i=1}^n\left(1-\frac1{[L(\sqrt[\ell]{\alpha_i}):L]}\right)
\prod_{p\mid m}\left(1-\frac1{[L(\sqrt[p]{\epsilon}):L]}\right),
\]
and this is positive provided that $\zeta_{16}$,
$\sqrt[\ell]{\alpha_i}$, and $\sqrt[p]{\epsilon}$ for $p\mid m$
are not contained in $L$.

To that end, first note that none of
$\sqrt[4]{\epsilon}$, $\sqrt[\ell]{\alpha_i}$, or
$\sqrt[p]{\epsilon}$ for $p\mid m$
is contained in $K(\zeta_r)$ for any $r$,
since they have both real and non-real conjugates, while
$K(\zeta_r)$ is an abelian extension of $\Q$.
Combining this with a consideration of degrees, we find that
$\sqrt[\ell]{\alpha_i}$ and $\sqrt[p]{\epsilon}$
are not contained in $L$.

Next we aim to show that $\zeta_{16}\notin L$.
Suppose otherwise, and let $d\mid m$ be the smallest number
such that $\zeta_{16}\in K(\zeta_{8d},\sqrt[8]{\epsilon})$.
Then for any prime $p\mid d$, $\zeta_{16}$ is not
contained in $M\coloneq K(\zeta_{8d/p},\sqrt[8]{\epsilon})$
but is contained in $M(\zeta_{p^{v_p(d)}})$. Therefore
$\Disc(M(\zeta_{16})/M)\mid\Disc(M(\zeta_{p^{v_p(d)}})/M)$, and so
$p=2$. Since $m$ is odd this cannot happen, so we must have $d=1$ and
$\zeta_{16}\in K(\zeta_8,\sqrt[8]{\epsilon})$.

Suppose $K\ne\Q(\sqrt2)$, and let $F=K(\zeta_8)$.
Let $k\in\{2,4\}$ be
the smallest power of $2$ such that $\sqrt[k]{\epsilon}\notin F$.
Since $K\ne\Q(\sqrt2)$, we have $\Disc{K}\nmid 16$, so
$K\not\subset\Q(\zeta_{16})$, and hence $\zeta_{16}\notin F$.
Therefore $F(\zeta_{16})=F(\sqrt{\zeta_8})$ is a quadratic extension of $F$
contained in $F(\sqrt[8]{\epsilon})$. Kummer theory implies that
$F(\sqrt[k]{\epsilon})$ is the unique such extension, and it follows that
$\sqrt[k]{\epsilon}\zeta_{16}\in F$.
In particular, $\sqrt\epsilon\in K(\zeta_{16})$.
Since $K\cap\Q(\zeta_{16})=K\cap\Q(\sqrt2)=\Q$,
an exercise in Galois theory shows that
\[
\Gal(K(\zeta_{16})/K)\cong\Gal(\Q(\zeta_{16})/\Q)\cong(\Z/16\Z)^\times
\cong C_4\times C_2,
\]
so $K(\zeta_{16})$ contains three quadratic extensions of $K$:
$K(\sqrt2)$, $K(i)$, and $K(i\sqrt{2})$.
Among these, only $K(\sqrt2)$ has a real embedding, so
we must have $K(\sqrt\epsilon)=K(\sqrt2)\subset F$,
meaning $k=4$.
Supposing that $\sqrt[4]{\epsilon}\zeta_{16}=a+bi\in F=K(\sqrt2,i)$
and taking the norm to $K(\sqrt2)$, we obtain $\sqrt\epsilon = a^2+b^2$.
This contradicts the fact that $\sqrt\epsilon$ is not
totally positive, so $\zeta_{16}\notin L$.

It remains only to consider $K=\Q(\sqrt2)$.
In this case the only quadratic extension of
$K$ contained in $K(\zeta_{16})=K\bigl(\sqrt{2+\sqrt2},i\bigr)$
with a real embedding is $K\bigl(\sqrt{2+\sqrt2}\bigr)$,
which cannot equal $K(\sqrt\epsilon)$ since $N(\epsilon)=-1$.
It follows that $[K(\zeta_{16},\sqrt[8]{\epsilon}):K(\zeta_8)]=16$,
so again $\zeta_{16}\notin L$.

\end{proof}

\begin{lemma}\label{lem:diophantine}\
\begin{itemize}
\item[(i)] There are no solutions to $x^2+1=y^r$ in positive integers with $r>1$.
\item[(ii)] If $x,y,r$ are positive integers satisfying $x^2+1=2y^r$ and $r>2$ then
either $x=y=1$ or $(x,y,r)=(239,13,4)$.
\end{itemize}
\end{lemma}
\begin{proof}
These are results of Lebesgue \cite{Lebesgue} and
St\"ormer \cite[\S3]{Stormer}, respectively.
\end{proof}

\section{Proof of main results}\label{sec:proof}
\begin{proof}[Proof of Theorem~\ref{thm:smallm}]
Consider $m\in\{1,2\}$ and $f\in\GEM(1,m)$.
By Gauss' Lemma, we can write
$f=c\prod_i f_i^{n_i}$, where $c\in\Z$ and $f_i\in\Z[x]$ is irreducible
with content $1$. It is easy to see from Definition~\ref{def:GEM}
that $F(x)=c^2\prod_i(f_i(0)f_i(x))$ is a
$\GEM(1,m)$ polynomial satisfying $\GEM_2(F;1,m)=\GEM_2(f;1,m)$.
Thus, replacing $f$ by $F$ if necessary,
we may assume without loss of generality that $f$ is
squarefree and $f(0)$ is a square.
 
Write $\GEM_2(f;1,m)=(p_n)_{n\ge1}$, and suppose that
it contains all primes except for a finite set $\{q_1,\ldots,q_r\}$.
(Note that the exceptional set must contain all primes dividing
$mf(0)$, but could otherwise be empty.)
Let $x>0$ be a large real number, let $p=p_{n+1}$ be the prime $\le x$
that occurs last in the sequence, and set $N=|f(p_1\cdots p_n)|$.
Then
\[
N=q_1^{e_1}\cdots q_r^{e_r}p^e,
\]
for some exponents $e_i\ge0$ and $e>0$.

Suppose first that $\deg{f}\ge3$. Then
applying Lemma~\ref{lem:siegel} to both $f$ and $-f$,
we see that $N$ is not a square when $x$ is sufficiently large,
so by Lemma~\ref{lem:nonsquare} there is a prime $P$ exceeding
$q_1,\ldots,q_r$ such
that $\left(\frac{P}{N}\right)=-1$. Next we apply
Theorem~\ref{thm:quadchars} to the characters
\[
\left(\frac{\cdot}{q_1}\right),
\ldots,
\left(\frac{\cdot}{q_r}\right),
\left(\frac{\cdot}{p}\right),
\left(\frac{-4}{\cdot}\right).
\]
This produces a squarefree positive integer $d\equiv1\pmod4$ such that
\[
\left(\frac{d}{q_i}\right)=\left(\frac{P}{q_i}\right),
\quad
\left(\frac{d}{p}\right)=\left(\frac{P}{p}\right)
\]
and every prime factor of $d$ is
$O_\varepsilon\bigl((q_1\cdots q_rp)^{\frac1{4\sqrt{e}}+\varepsilon}\bigr)$.
Since $\frac1{4\sqrt{e}}<1$, this is less than $p$ for sufficiently
large $x$, whence $d\mid p_1\cdots p_n$.

By multiplicativity of the Kronecker symbol, it follows that
$\left(\frac{d}{N}\right)=\left(\frac{P}{N}\right)=-1$. However,
by our construction of $d$, we have $N\equiv\pm f(0)\pmod{d}$,
so by quadratic reciprocity,
\[
\left(\frac{d}{N}\right)
=\left(\frac{N}{d}\right)=\left(\frac{\pm f(0)}{d}\right)=1.
\]
This is a contradiction, completing the proof when $\deg{f}\ge3$.

Next suppose $\deg{f}=2$, and write $f=ax^2+bx+c$
with $c$ a square, and $D\coloneq b^2-4ac\ne0$.
If $a$ is also a square then it follows from Lemma~\ref{lem:squarevalues}
that $|f|$ has at most finitely many square values, so the above proof
goes through. Hence we may assume that $\sqrt{a}\notin\Q$.
If $a<0$ then, choosing $p_i\equiv3\pmod{4}$,
for large enough $x$ it follows that
\[
\left(\frac{N}{p_i}\right)
=\left(\frac{-f(p_1\cdots p_n)}{p_i}\right)
=\left(\frac{-c}{p_i}\right)
=-1.
\]
Thus $N$ is not a square, and again the proof goes through.
Hence we may assume that $a>0$.

Suppose $f(u)=v^2$ for some $u,v\in\Z$. Completing the square, we have
\begin{equation}\label{eq:Pell}
(2au+b)^2 - 4av^2 = D.
\end{equation}
Let $K=\Q(\sqrt{a})$, with ring of integers $\OK$ and
fundamental unit $\epsilon\in\OK^\times$, and
let $\sigma:K\to K$ denote the non-trivial automorphism.
By the theory of Pell's equation,
there are elements $\rho_1,\ldots,\rho_s\in\OK$
such that
\begin{itemize}
\item $N(\rho_i)=D$ for each $i$;
\item $\rho_i/\rho_j\notin\OK^\times$ for $i\ne j$;
\item for any solution to \eqref{eq:Pell}, there exist
$i\in\{1,\ldots,s\}$ and $k\in\Z$ such that $N(\epsilon)^k=1$ and
\[
2au+b+v\sqrt{4a} = \pm\rho_i\epsilon^k.
\]
\end{itemize}

Returning to the setup above,
we wish to rule out solutions to \eqref{eq:Pell} with $u=p_1\cdots p_n$.
If the desired conclusion is false then we can generate infinitely
many such solutions by increasing $x$. In each solution we may
assume that $k\ge0$ by applying $\sigma$ if necessary, and
by the pigeonhole principle
we can choose infinitely many solutions with a fixed value of $\rho_i=\rho$
and a fixed $\pm$ sign. Replacing $\rho$ by $-\rho$ if necessary, we thus obtain
\[
4ap_1\cdots p_n=\rho\epsilon^k+\rho^\sigma\epsilon^{-k}-2b
\]
for infinitely many $n$ (with $k$ depending on $n$, but $a,b,\rho,\epsilon$
fixed).
Set $\rho_0=b+\sqrt{4ac}\in\OK$, and note that $N(\rho_0)=D$
and $\rho_0+\rho_0^\sigma=2b$, so the above can be written as
\[
4ap_1\cdots p_n=
\rho\epsilon^k-\rho_0 + \rho^\sigma\epsilon^{-k}-\rho_0^\sigma
=\frac{(\rho\epsilon^k - \rho_0)(\rho\epsilon^k - \rho_0^\sigma)}{\rho\epsilon^k}.
\]
Thus if $p_i$ is a prime not dividing $D$ then
for each prime $\pp\mid p_i\OK$ we have
\[
\epsilon^k \equiv \rho^{-1}\rho_0\text{ or }\rho^{-1}\rho_0^\sigma\pmod{\pp}.
\]

Suppose first that $\rho\notin\rho_0\OK^\times\cup\rho_0^\sigma\OK^\times$.
Then $\alpha=\rho^{-1}\rho_0$ and $\beta=\rho^{-1}\rho_0^\sigma$
are non-integral elements of $K$ with norm $1$, each with negative valuation at
a prime dividing $\rho\OK$. Choose a prime $\ell>2$
that does not divide either valuation, and apply the pigeonhole
principle again to select $k$ in a fixed residue class
$k_0\bmod\ell$.
Then neither $\alpha\epsilon^{-k_0}$ nor $\beta\epsilon^{-k_0}$
is an $\ell$th power in $K$, but for every $\pp\mid p_i\OK$
for some $p_i\nmid D$,
one of them is an $\ell$th power modulo $\pp$.
By Lemma~\ref{lem:chebotarev}, there are infinitely many
$\pp$ for which this does not happen, and that
results in a contradiction for large enough $n$.

Hence we have
$\rho\in\rho_0\OK^\times\cup\rho_0^\sigma\OK^\times$,
and replacing $\rho_0$ by $\rho_0^\sigma$ if necessary
we may assume that $\rho\in\rho_0\OK^\times$. Further,
replacing $k$ by a suitable translate, we may assume
that $\rho=\pm\rho_0$, so we obtain
\[
4ap_1\cdots p_n=\epsilon^{-k}(\epsilon^k \mp 1)(\pm\rho_0\epsilon^k - \rho_0^\sigma).
\]

Consider first the case when $\rho_0^{-1}\rho_0^\sigma\notin\OK^\times$.
Then applying Lemma~\ref{lem:chebotarev} in a similar manner,
we can find infinitely many primes $\pp$ such that
$-\epsilon^{-k_0}\rho_0^{-1}\rho_0^\sigma$ is not
an $\ell$th power modulo $\pp$ and $\epsilon$ has odd order modulo $\pp$.
The latter condition implies that $-1$ is not congruent modulo $\pp$
to a power of $\epsilon$, so this again results in a contradiction
if $\rho=-\rho_0$.
Therefore we have $\rho=\rho_0$ and
\[
4ap_1\cdots p_n=\epsilon^{-k}(\epsilon^k-1)(\rho_0\epsilon^k - \rho_0^\sigma).
\]
Let $\qq$ be a prime with $\qq\nmid 2aDq_1\cdots q_r\OK$
such that $\epsilon\bmod\qq$ has odd order, which we can ensure
by Lemma~\ref{lem:chebotarev}.
Let $m$ be the order of $\epsilon\bmod\qq^2$, which
must also be odd.
By Lemma~\ref{lem:chebotarev},
there are infinitely many primes $\pp$ such that
$\epsilon^{-k_0}\rho_0^{-1}\rho_0^\sigma$ is not
an $\ell$th power modulo $\pp$ and the order of $\epsilon\bmod\pp$
is divisible by $m$.
It follows that $m\mid k$, which in turn implies that
$\qq^2\mid(\epsilon^k-1)\OK$. Since $p_1\cdots p_n$ is squarefree,
this is a contradiction.

Suppose now that $\rho_0^{-1}\rho_0^\sigma\in\OK^\times$.
Then it follows that $|D|$ is a square and $\rho_0\in\sqrt{|D|}\OK^\times$,
so we have $|\rho_0|=\sqrt{|D|}\epsilon^j$ for some $j\in\Z$.
Thus, up to a constant factor, we
need to analyse the reduction mod $\pp$ of
$(\epsilon^k\pm1)(\epsilon^{k+2j}\pm\sgn{D})$.
Applying Lemma~\ref{lem:chebotarev} again, we see that
the $(\epsilon^k+1)(\epsilon^{k+2j}+1)$
case cannot occur, and in the remaining cases
we conclude that $m$ must divide one of $k$ or
$k+2j$. Since $m$ is odd, this still suffices to
imply that $\qq^2\mid p_1\cdots p_n\OK$, completing the
proof when $\deg{f}=2$.

Finally suppose $\deg{f}=1$, and write $f=ax+b$ with $b$ a square.
The case when $a<0$ can be handled by reducing modulo a prime
$p_i\equiv3\pmod4$ as in the proof for degree $2$ above,
so we may assume that $a>0$.
Let $g=(a,b)$, and set $a'=a/g$, $b'=b/g$.
By our construction, $p_1\cdots p_n$ is coprime
to $b$, and it follows that
$N'\coloneq\frac{N}{g}=a'p_1\cdots p_n+b'$
is coprime to $b'$. Suppose $b'N'$ is not an $8$th power,
and note that this must be the case if $2\mid a'$ and $32\nmid a'$.
Let $m\in\{2,4,8\}$ be the least power of $2$ such that
$b'N'$ is not an $m$th power, and set
$y=(b'N')^{\frac2m}\in\N$. Then
$y$ is not a square, so by Lemma~\ref{lem:nonsquare} there is a prime
$P$ exceeding $q_1,\ldots,q_r$ for which $\left(\frac{P}{y}\right)=-1$.
Applying Theorem~\ref{thm:quadchars} as before,
we obtain a squarefree number $d$ such that
\[
\left(\frac{d}{q_i}\right)=\left(\frac{P}{q_i}\right),
\quad
\left(\frac{d}{p}\right)=\left(\frac{P}{p}\right),
\quad
\left(\frac{-4}{d}\right)=1\text{ when }m=2,
\]
and every prime factor of $d$ is
$O_{m,\varepsilon}\bigl((q_1\cdots q_rp)^{\frac{\varphi(m)}{4\sqrt{e}}+\varepsilon}\bigr)$ and congruent to $1\pmod{m}$.
Since $\frac{\varphi(m)}{4\sqrt{e}}<1$, it follows that
$d\mid p_1\cdots p_n$ when $x$ is sufficiently large.
For $p_i\equiv1\pmod{m}$, Euler's criterion implies
\[
\left(\frac{y}{p_i}\right)\equiv
y^{\frac{p_i-1}{2}}=(b'N')^{\frac{p_i-1}{m}}
\equiv\bigl((b')^2\bigr)^{\frac{p_i-1}{m}}
=\bigl((b')^{\frac2m}\bigr)^{p_i-1}\equiv1\pmod{p_i}.
\]
Since all prime factors of $b'N'$ are contained in $\{q_1,\ldots,q_r,p\}$,
we thus have
\[
-1=\left(\frac{P}{y}\right)
=\left(\frac{d}{y}\right)
=\left(\frac{y}{d}\right)
=1,
\]
which is again a contradiction.

Hence we may assume that $b'$ is an $8$th power and $a'$ is odd or divisible
by $32$. In this case we give a different proof based on ERH.
Let $\ell$ be the smallest odd prime that occurs in
$\GEM_2(f;1,m)$. Replacing $f$ by $b^{\ell-1}f$, we may assume that $b$
is an $\ell$th power.

Assume that $x$ is large enough to ensure that
$\ell\mid p_1\cdots p_n$.
Then we have $N\equiv b\pmod{\ell}$
but $N\not\equiv b\pmod{\ell^2}$
since $\ell^2\nmid ap_1\cdots p_n$.
Therefore, if $\xi\pmod{\ell^2}$ is a character
of order $\ell$ then $\xi(N)\ne1$.
Assuming ERH, we can apply Theorem~\ref{thm:orderl} to choose a character
$\chi$ of squarefree conductor $d$ and order dividing $\ell$ such that
\[
\chi(p)=\xi(p)
\quad\text{and}\quad
\chi(q_i)=\xi(q_i)
\text{ for }i=1,\ldots,r,
\]
and every prime factor of $d$ is $O(\ell^4\log^2(\ell q_1\cdots q_rp))$.
For sufficiently large $x$ this is less than $p$, and therefore
$d\mid p_1\cdots p_n$, so that $\chi(N)=\chi(b)=1$.
However, by construction we have
\[
\chi(N)=\chi(q_1)^{e_1}\cdots\chi(q_r)^{e_r}\chi(p)^e
=\xi(q_1)^{e_1}\cdots\xi(q_r)^{e_r}\xi(p)^e=\xi(N)\ne1.
\]
This is a contradiction, and that completes the proof
of Theorem~\ref{thm:smallm}.
\end{proof}

\begin{proof}[Proof of Theorems~\ref{thm:ERH} and \ref{thm:unconditional}]
Let $m\in\N$ and $c\in S(m)$. We may assume that $m\ne2$, and that
$m=1\implies c\in\{-5\}\cup\{1-2^k:k\in\N\}$, since
Theorem~\ref{thm:ERH} and all other cases of Theorem~\ref{thm:unconditional}
for $m\le2$ follow from Theorem~\ref{thm:smallm}.
As in the proof above we suppose that
$\GEM_2(\Phi_m(cx);1,m)=(p_n)_{n\ge1}$ contains all primes
$\equiv1\pmod{m}$ except for a finite (possibly empty) set
$\{q_1,\ldots,q_r\}$. Let $\mu$ be the largest prime factor of
$m$ if $m>1$, and set $\mu=1$ if $m=1$.
Let $x>0$ be a large real number, let $p=p_{n+1}$
be the prime $\le x$ that occurs last in the sequence, and set
$M=|\Phi_m(cp_1\cdots p_n)|$.
Then we have
\[
M = \mu^\delta q_1^{e_1}\cdots q_r^{e_r}p^e
\]
for some exponents $\delta\in\{0,1\}$, $e_i\ge0$, and $e>0$,
with $\delta=0$ when $m=1$.

Set $N=M/\mu^\delta$.
Then we claim that $N$ is not a square for $x$
sufficiently large:
\begin{itemize}
\item For $m>2$ and $\delta=0$,
the claim follows from Lemma~\ref{lem:squarevalues}
applied to $\Phi_m$.
\item For $m\notin\{3,4,6\}$ and $\delta=1$,
the claim follows from Lemma~\ref{lem:siegel} applied to $\mu\Phi_m$.
\item For $m\in\{3,4,6\}$ and $\delta=1$, for large enough $x$ we can choose
$p_i$ for $i\le n$ such that $\left(\frac{\mu}{p_i}\right)=-1$
(viz.\ $p_i\equiv5\pmod8$ for $m=4$, $p_i\equiv7\pmod{12}$ for $m\in\{3,6\}$).
Since $M\equiv1\pmod{p_i}$, it follows that $\left(\frac{N}{p_i}\right)=-1$.
\item For $m=1$ and
$c\in\{-5\}\cup\bigl\{1-2^k:k\in\N\bigr\}$, $2$ occurs in the sequence
as either the first or second term, after
which point we have $N\equiv3\pmod4$.
\end{itemize}

Applying Lemma~\ref{lem:nonsquare}, there exists a prime $P$ exceeding
$q_1,\ldots,q_r$ such that $\left(\frac{P}{N}\right)=-1$.
Next we apply Theorem~\ref{thm:quadchars} to the characters
\[
\left(\frac{\cdot}{q_1}\right),
\ldots,
\left(\frac{\cdot}{q_r}\right),
\left(\frac{\cdot}{p}\right),
\]
and also $\left(\frac{-4}{\cdot}\right)$ when $4\nmid m$
and $\left(\frac{8}{\cdot}\right)$ when $m=4$.
(Note that we allow $q_i=2$ when $m=1$, in which case
$\left(\frac{\cdot}{q_i}\right)=\left(\frac{8}{\cdot}\right)$
is a character of conductor $8$.)
This produces a squarefree $d\equiv1\pmod{4}$ such that
\[
\left(\frac{d}{q_i}\right)=\left(\frac{P}{q_i}\right),
\quad
\left(\frac{d}{p}\right)=\left(\frac{P}{p}\right),
\quad
\left(\frac{d}{2}\right)=1\text{ when }m=4,
\]
and every prime factor of $d$ is
$O_{m,\varepsilon}\bigl((q_1\cdots q_rp)^{\frac{\varphi(m)}{4\sqrt{e}}+\varepsilon}\bigr)$
and congruent to $1\pmod{m}$.
When $m\in\{1,\ldots,10,12,14,18\}$ we have
$\frac{\varphi(m)}{4\sqrt{e}}<1$, so for sufficiently large
$x$ every prime factor of $d$ is less than $p$, whence $d\mid p_1\cdots p_n$.
For general $m$ the same conclusion holds under the assumption of GLH.

By construction we have $M\equiv\pm1\pmod{d}$ and $d\equiv1\pmod{m}$,
and thus
\[
-1=\left(\frac{P}{N}\right)=
\left(\frac{d}{N}\right) = \left(\frac{N}{d}\right)
=\left(\frac{\pm \mu^\delta}{d}\right)
=\left(\frac{d}{\mu^\delta}\right)=1.
\]
This is a contradiction, and that completes the proof.
\end{proof}

\begin{proof}[Proof of Theorem~\ref{thm:firstprime}]
Let $m=2^k$ for some $k>1$, and write $\GEM_2(\Phi_m(cx);1,m)=(p_n)_{n\ge1}$.
Let $p$ be the smallest prime congruent
to $1\pmod{m}$, and suppose that $p=p_{n+1}$ for some $n\ge0$.

Suppose $c$ is even. Then we have
\[ (cp_1\cdots p_n)^{2^{k-1}} + 1 = p^r. \]
By Lemma~\ref{lem:diophantine}(i), this can only happen if $r=1$, which forces
$n=0$.

On the other hand, if $c$ is odd then we have
\[ (cp_1\cdots p_n)^{2^{k-1}} + 1 = 2p^r. \]
Noting that $239$ is a prime congruent to $3\pmod4$,
Lemma~\ref{lem:diophantine}(ii) implies that if $n>0$ then we must have $r\le2$,
which is a contradiction. Thus we again have $n=0$, and either
$(k,c,p,r)=(2,\pm239,13,4)$ or $r\le2$.
However, if $k=2$ then $p=5$, so the former case does not apply.

Hence, in all cases we have
\begin{equation}\label{eq:firstprime}
c^{2^{k-1}} + 1 \in \bigl\{p,2p,2p^2\bigr\},
\end{equation}
and it follows that $2p^2>3^{2^{k-1}}$.
Set $x = 3^{2^{k-2}}/\sqrt{2}$. Then by \cite[Theorem~1.2]{BMOR},
\[
(80\log{3}-1)2^{k-1} > 80\log{2}
\implies \theta(x;2^k,1) > \frac{x}{2^{k-1}}-\frac{x}{160\log{x}}>0,
\]
provided that $2^k>80$ and $x>\exp(0.03\cdot 2^{k/2}\log^3(2^k))$,
which holds for $k\ge12$. Noting that $17\equiv1\pmod{2^4}$, $97\equiv1\pmod{2^5}$, and $12289\equiv1\pmod{2^{12}}$ are all prime, we see that $2p^2<3^{2^{k-1}}$
whenever $k\ge4$.

Hence we must have $(k,p)=(2,5)$ or $(3,17)$, and then it is straightforward
to see that the only solutions to \eqref{eq:firstprime} are
$k=2$, $c\in\{\pm2,\pm3\}$ and $k=3$, $c\in\{\pm2\}$.
In each of those cases, $p$ occurs as $p_1$, but
then the next prime $\equiv1\pmod{2^k}$ is omitted, as we can see
by applying Lemma~\ref{lem:diophantine} to the equations
\[
(\pm10p_2\cdots p_n)^2+1=13^r,
\quad
(\pm15p_2\cdots p_n)^2+1=2\cdot13^r,
\quad
(\pm34p_2\cdots p_n)^4+1=41^r.
\]
\end{proof}

\bibliographystyle{amsplain}
\bibliography{paper}
\end{document}